\documentclass{proc-l}

\usepackage{amssymb,url}

\usepackage{graphicx}

\newtheorem{theorem}{Theorem}[section]
\newtheorem{lemma}[theorem]{Lemma}

\theoremstyle{definition}
\newtheorem{definition}[theorem]{Definition}

\theoremstyle{remark}

\numberwithin{equation}{section}

\usepackage{tikz}
\definecolor{mygray}{rgb}{0.3333,0.3333,0.3333} 

\def\myint#1#2#3{\int_{#1} #2\,\mathrm{d}#3} 

\newcommand{\R}{\ensuremath{\mathbb{R}}} 
\newcommand{\N}{\ensuremath{\mathbb{N}}}

\newcommand{\supp}{\mathrm{supp}\,}

\begin{document}

 \title[Local uniqueness for an inverse BVP with partial data]{Local uniqueness for an inverse boundary value problem with partial data}

 \author{Bastian Harrach}
\address{University of Stuttgart\\
Department of Mathematics - IMNG\\
Chair of Optimization and Inverse Problems\\
Allmandring 5b\\
70569 Stuttgart\\
Germany}
\email{harrach@math.uni-stuttgart.de}
\thanks{The authors would like to thank the German Research Foundation (DFG) for financial support of the project within the
Cluster of Excellence in Simulation Technology (EXC 310/1) at the University of Stuttgart.}
\thanks{This is a prepublication version of a journal article published in \emph{Proc.\ Amer.\ Math.\ Soc.}\ \textbf{145}(3), 1087--1095, 2017, 
\url{https://doi.org/10.1090/proc/12991}.
}

\author{Marcel Ullrich}
\address{University of Stuttgart\\
Department of Mathematics - IMNG\\
Chair of Optimization and Inverse Problems\\
Allmandring 5b\\
70569 Stuttgart\\
Germany}
\email{marcel.ullrich@mathematik.uni-stuttgart.de}


\date{\today}


\begin{abstract}
In dimension $n\geq 3$, we prove a local uniqueness result for the potentials $q$ of the Schr\"odinger equation $-\Delta u+qu=0$ from partial boundary data.
More precisely, we show that potentials $q_1,q_2\in L^\infty$ with positive essential infima can be distinguished by local boundary data
if there is a neighborhood of a boundary part where $q_1\geq q_2$ and $q_1\not\equiv q_2$.
\end{abstract}

\maketitle

\section{Introduction}\label{sec:intro}

Let $\Omega\subseteq\mathbb{\R}^n$, $n\geq 3$, be a Lipschitz-domain with outer normal $\nu$ and $L^\infty_+(\Omega)$ denote the subspace
of $L^\infty(\Omega)$-functions with positive essential infima.
We consider the question whether the potential $q\in L^\infty_+(\Omega)$ in the Schr\"odinger equation
\begin{equation}\label{eq:elliptic_schroedinger_equation_1}
 -\Delta u +qu=0\quad\text{in}\quad\Omega
\end{equation}
is uniquely determined by partial boundary data
on a possibly arbitrarily small non-empty relatively open subset $\Gamma\subseteq\partial\Omega$.

For such a boundary subset $\Gamma\subseteq\partial\Omega$ and $q\in L^\infty_+(\Omega)$, the partial boundary data, we consider in this work, is given by the \emph{local Neumann-to-Dirichlet (NtD) operator}
\begin{equation}
 \Lambda_\Gamma(q)\,:\,L^2(\Gamma)\to L^2(\Gamma),\quad g\mapsto u^{(g)}_q\vert_\Gamma,
\end{equation}
where $u^{(g)}_q\in H^1(\Omega)$ is the solution of
\begin{equation}\label{eq:N_BVP_of_schroedinger_equation}
    -\Delta u^{(g)}_q +qu^{(g)}_q=0\ \text{in}\ \Omega
   \quad\quad\text{with}\quad\quad\partial_\nu u^{(g)}_q\vert_{\partial\Omega}=
  \begin{cases}
   g,&\text{on}\ \Gamma,\\
   0,&\text{on}\ \partial\Omega\setminus \Gamma.
  \end{cases}
\end{equation}
$\Lambda_\Gamma(q)$ is easily shown to be a compact self-adjoint linear operator.

In this article, we will show the following local uniqueness result.

\begin{theorem}\label{th:main_theorem_of_distinguishability} 
Let $q_1,q_2\in L^\infty_+(\Omega)$ and $V\subseteq\mathbb{R}^n$ be an open connected set with
$q_1\geq q_2$ on $\Omega\cap V$ and $\Gamma:=\partial\Omega\cap V\neq\emptyset$. Then,
\begin{equation}
 q_1\vert_{\Omega\cap V}\not\equiv q_2\vert_{\Omega\cap V}\quad\text{implies}\quad\Lambda_\Gamma(q_1)\neq\Lambda_\Gamma(q_2).
\end{equation}
Moreover, in that case $\Lambda_\Gamma(q_2)-\Lambda_\Gamma(q_1)$ has a positive eigenvalue.
\end{theorem}

The \emph{inverse potential problem} of the Schr\"odinger equation is closely related to the \emph{inverse conductivity problem (Calder\'on Problem \cite{calderon1980inverse,calderon2006inverse})}.
For both problems, uniqueness from full boundary data on $\partial\Omega$ has been extensively studied in the last 30 years. 
To give a brief overview of prominent contributions, we list Kohn and Vogelius \cite{kohn1984determining,kohn1985determining}, Sylvester and Uhlmann \cite{sylvester1987global}, Nachman \cite{nachman1996global},
Astala and P\"aiv\"arinta \cite{astala2006calderon}, Bukhgeim \cite{bukhgeim2008recovering}, Haberman and Tataru \cite{haberman2013uniqueness}.

The uniqueness problem from partial boundary data has attracted growing attention over the last years.
Typically, this is studied for data of type $C^D_q$ or $C^N_q$ on sets $\Gamma^D,\Gamma^N\subseteq\partial\Omega$, where
\begin{equation*}
 C^D_q :=\left\lbrace \left(u\vert_{\Gamma^D},\partial_\nu u\vert_{\Gamma^N}\right)\,:\,-\Delta u + q u=0,\ \supp(u\vert_{\partial\Omega})\subseteq \Gamma^D\right\rbrace,
\end{equation*}
\begin{equation*}
 C^N_q :=\left\lbrace \left(u\vert_{\Gamma^D},\partial_\nu u\vert_{\Gamma^N}\right)\,:\,-\Delta u + q u=0,\ \supp(\partial_\nu u\vert_{\partial\Omega})\subseteq \Gamma^N\right\rbrace.
\end{equation*}
Obviously, for potentials $q\in L^\infty_+(\Omega)$, the question of uniqueness from data of type $C^N_q$ with $\Gamma=\Gamma^D=\Gamma^N$ is equivalent to the question of uniqueness from the
local NtD operator $\Lambda_\Gamma(q)$.

Hereafter, we list some recent results. Let us also refer to the article of Kenig and Salo \cite{kenig2014recent} for a further overview.  

For dimension $n=2$, Imanuvilov, Uhlmann and Yamamoto showed uniqueness from data of type $C^D_q$ in \cite{imanuvilov2010calderon}, where $\Gamma^D=\Gamma^N$ is an arbitrary open subset in $\partial\Omega$ and the
potentials are in $\mathcal{C}^{2+\alpha}\left(\overline{\Omega}\right)$ for $\alpha>0$.

For dimension $n\geq 3$, Kenig, Sj\"ostrand and Uhlmann proved uniqueness from data of type $C^D_q$ for $q\in L^\infty(\Omega)$ in \cite{kenig2007calderon}, where $\Gamma^D$ and $\Gamma^N$ are open
neighborhoods slightly larger than a front face and a back face of $\partial\Omega$, respectively.
Nachman and Street presented a constructive proof of this result in \cite{nachman2010reconstruction}.
In \cite{isakov2007uniqueness}, Isakov proved uniqueness from data of type $C^D_q$ for $q\in L^\infty(\Omega)$ assuming $\Gamma^D=\Gamma^N$ and that the remaining boundary part is contained in a plane or a sphere.
In \cite{kenig2012calderon}, Kenig and Salo presented a result that unifies
and improves the approaches of \cite{kenig2007calderon} and \cite{isakov2007uniqueness}. In particular, they reduced the assumptions regarding the sets $\Gamma^D$, $\Gamma^N$ and, if so (in \cite{isakov2007uniqueness}),
the remaining boundary part. Let us refer to their article for a detailed description.

Theorem \ref{th:main_theorem_of_distinguishability} is, to the knowledge of the authors, the first result that presents a uniqueness result for partial data
on an arbitrary non-empty relatively open boundary part $\Gamma\subseteq\partial\Omega$ (with $\Gamma=\Gamma^D=\Gamma^N$) for dimension $n\geq 3$.
Except the assumption that $\Omega$ has to be a Lipschitz-domain, there are no further assumptions to the boundary required: neither to the boundary part $\Gamma$ nor to the remaining boundary part.

This paper is organized as follows. In Section \ref{sec:proof_of_the_main_theorem}, we prove Theorem \ref{th:main_theorem_of_distinguishability}.
For this purpose, we present and combine a \emph{monotonicity relation} for the local NtD operator (Lemma \ref{lemma:monotonicity}) and
a new variant of the concept of \emph{localized potentials} (Lemma \ref{lemma:localized_potentials}, cf.\ \cite{gebauer2008localized} for the initial concept).
%
Lemma \ref{lemma:monotonicity} presents a monotonicity inequality that yields a lower bound for the change of the local NtD operator (represented by its corresponding quadratic form) caused by a potential
change.
This lower bound depends on the spatial change of the potential weighted by the solution of the Schr\"odinger equation for the initial potential.  
%
Lemma \ref{lemma:localized_potentials} shows a possibility to control the lower bound of the monotonicity inequality.\footnote{Originally,
the concept of localized potentials was used to locally control electrical potentials for the inverse conductivity problem. Since in this work it is used to locally weight
the potentials of the Schr\"odinger equation, it seems appropriate to keep with the name ``localized potentials''.}
The approach of combining a monotonicity relation with the concept of localized potentials has previously been used
in \cite{harrach2013monotonicity,harrach2009uniqueness,harrach2012simultaneous}.
The proofs of the Lemmas \ref{lemma:monotonicity} and \ref{lemma:localized_potentials} are postponed to
the Sections \ref{sec:monotonicity_for_the_NtDs} and \ref{sec:localized_potentials}, respectively.

\section{The proof of Theorem \ref{th:main_theorem_of_distinguishability}}\label{sec:proof_of_the_main_theorem}

Let $q_1,q_2\in L^\infty_+(\Omega)$, let $V\subseteq\mathbb{R}^n$ be an open connected set and let $\Gamma:=\partial\Omega\cap V\neq\emptyset$.

To prove Theorem \ref{th:main_theorem_of_distinguishability}, we combine a monotonicity inequality (Lemma \ref{lemma:monotonicity}) for Neumann-to-Dirichlet operators
with a result about the existence of localized potentials (Lemma \ref{lemma:localized_potentials}).

\begin{lemma}\label{lemma:monotonicity}
 Let $g\in L^2(\Gamma)$ and $u_1:=u^{(g)}_{q_1}\in H^1(\Omega)$ be the corresponding solution of \eqref{eq:N_BVP_of_schroedinger_equation}. Then,
\begin{equation}
 \left(g,\left(\Lambda_\Gamma(q_2)-\Lambda_\Gamma(q_1)\right)g\right)_{L^2(\Gamma)}\geq-\myint{\Omega}{(q_2-q_1){u_1}^2}{x}.
\end{equation}
 \end{lemma}
 
 Lemma \ref{lemma:monotonicity} is proven in Section \ref{sec:monotonicity_for_the_NtDs}.

\begin{lemma}\label{lemma:localized_potentials}
Let $q_1\gneq q_2$ on $\Omega\cap V$ (i.e., $q_1\vert_{\Omega\cap V}\geq q_2\vert_{\Omega\cap V}$ and $q_1\vert_{\Omega\cap V}\not\equiv q_2\vert_{\Omega\cap V}$).
Then, there exists a sequence $(g_m)_{m\in\N}\subset L^2(\Gamma)$ such that the
corresponding solutions $\left(u_m\right)_{m\in\N}:=\left(u^{(g_m)}_{q_1}\right)_{m\in\N}\subset H^1(\Omega)$ of \eqref{eq:N_BVP_of_schroedinger_equation}
fulfill
 \begin{equation}
  \lim_{m\to\infty}\myint{V\cap\Omega}{(q_1-q_2){u_m}^2}{x}=\infty\quad\text{and}\quad\lim_{m\to\infty}\myint{\Omega\setminus V}{(q_1-q_2){u_m}^2}{x}=0.
 \end{equation}
\end{lemma}

Lemma \ref{lemma:localized_potentials} is proven in Section \ref{sec:localized_potentials}.
%
%
%

\begin{proof}[Proof of Theorem \ref{th:main_theorem_of_distinguishability}]
First, we apply Lemma \ref{lemma:localized_potentials}: There exists a $g\in L^2(\Gamma)$
such that the corresponding solution
$u:=u^{(g)}_{q_1}$ of \eqref{eq:N_BVP_of_schroedinger_equation}
fulfills
\begin{equation*}
   \myint{V\cap\Omega}{(q_1-q_2)u^2}{x}>1\quad\text{and}\quad\myint{\Omega\setminus V}{(q_1-q_2)u^2}{x}>-1.
\end{equation*}
Now, we apply Lemma \ref{lemma:monotonicity} and obtain
\begin{align*}
 \left(g,\left(\Lambda_\Gamma(q_2)-\Lambda_\Gamma(q_1)\right)g\right)_{L^2(\Gamma)}&\geq-\myint{\Omega}{(q_2-q_1){u}^2}{x}\\
								      &=\myint{V\cap\Omega}{(q_1-q_2){u}^2}{x}+\myint{\Omega\setminus V}{(q_1-q_2){u}^2}{x}\\
								      &> 1-1=0.
\end{align*}
This shows that $\Lambda_\Gamma(q_2)-\Lambda_\Gamma(q_1)$ is not semi negative definite and thus has a positive eigenvalue.
\end{proof}

\section{Monotonicity for Neumann-to-Dirichlet maps}\label{sec:monotonicity_for_the_NtDs}

Again, let $q_1,q_2\in L^\infty_+(\Omega)$, $V\subseteq\mathbb{R}^n$ be an open connected set and $\Gamma:=\partial\Omega\cap V\neq\emptyset$.
%

Such monotonicity estimates are well-known for the inverse conductivity problem, cf., e.g., Ikehata, Kang, Seo, and Sheen \cite{Kan97,ikehata1998size}.

Lemma \ref{lemma:monotonicity} follows from \cite[Lemma 4.1]{harrach2009uniqueness}. Since the proof is simple and short, we include it for the sake of completeness.

\begin{proof}[Proof of Lemma \ref{lemma:monotonicity}]
Let $g\in L^2(\Gamma)$ and $u_i:=u^{(g)}_{q_i}\in H^1(\Omega)$ be the corresponding solutions of \eqref{eq:N_BVP_of_schroedinger_equation} for $i\in\lbrace 1,2\rbrace$. Then,
\begin{equation*}
 b_i(u_i,w):=\myint{\Omega}{\nabla u_i \nabla w + q_i u w}{x}=\myint{\Gamma}{g w\vert_\Gamma}{s}=:l(w)\quad\forall w\in H^1(\Omega),\ i\in\lbrace 1,2 \rbrace.
\end{equation*}

Now, we use this and consider
\begin{align*}
 &\left(g,\left(\Lambda_\Gamma(q_2)-\Lambda_\Gamma(q_1)\right)g\right)_{L^2(\Gamma)}\\
 &\quad\quad=l(u_2)-l(u_1)=b_2(u_2,u_2)-2 b_2(u_2,u_1)+b_1(u_1,u_1)\\
  \label{eq:deriving_monotonicity}
  &\quad\quad=-\myint{\Omega}{(q_2-q_1)u_1^2-\left(\nabla(u_2-u_1)\right)^2-q_2(u_1-u_2)^2}{x}.
\end{align*}
Since $q_2\geq0$, the assertion follows.
\end{proof}

\section{Localized Potentials}
\label{sec:localized_potentials}

Again, let $q_1,q_2\in L^\infty_+(\Omega)$, $V\subseteq\mathbb{R}^n$ be an open connected set and $\Gamma:=\partial\Omega\cap V\neq\emptyset$.
In addition, as assumed in Lemma \ref{lemma:localized_potentials}, let $q_1\gneq q_2$ on $\Omega\cap V$ (i.e., $q_1\vert_{\Omega\cap V}\geq q_2\vert_{\Omega\cap V}$ and $q_1\vert_{\Omega\cap V}\not\equiv q_2\vert_{\Omega\cap V}$).

%

Since the open set $V\cap\Omega$ is a countable union of closed balls and $q_1\gneq q_2$ on $\Omega\cap V$, there exists a closed ball
\begin{equation}\label{eq:the_ball_B}
 B\subseteq V\cap\Omega\quad \text{where}\quad  q_1\gneq q_2
\end{equation}
and $V\setminus B$ is connected.

To prove Lemma \ref{lemma:localized_potentials}, we introduce two operators in Definition \ref{def:virtual_measurement_operators} and present some properties of these operators and their adjoints in Lemma \ref{lemma:well_definedness_and_some_properties_of_the_vm_operators}.
In the proof of Lemma \ref{lemma:well_definedness_and_some_properties_of_the_vm_operators}, the following two theorems play a key role.

\begin{theorem}\label{theorem:14th_importent_propertie} Let $H_1,H_2$ be two Hilbert spaces, $L\in\mathcal{L}(H_1,H_2)$ and $h\in H_2$. Then,
\begin{equation}
 h\in\mathcal{R}(L)\quad\Leftrightarrow\quad\exists C>0\,:\,|(h,g)_{H_2}|\leq C\|L^\ast g\|_{H_1}\quad\forall g\in H_2.
\end{equation}
\end{theorem}

\begin{proof}
 This is a well-known result from functional analysis (see, e.g., the book of Bourbarki \cite{bourbaki2003elements}). For Banach spaces, a proof is given in \cite[Lemma 3.4]{fruhauf2007detecting}.
\end{proof}


\begin{theorem}[Unique continuation from sets of positive measure]\label{th:UCP_from_a_set_of_positive_measure}
 Let $\Omega'\in\mathbb{R}^m$, $m\geq 3$, be a connected open set and $q\in L^\infty(\Omega')$. The trivial solution of
 \begin{equation}\label{eq:Schroedinger_equation_UCPpos}
  -\Delta u +qu=0
 \end{equation}
 is the only $H^1(\Omega')$-solution vanishing on a measurable set of positive measure.
\end{theorem}

\begin{proof}
%
 Theorem \ref{th:UCP_from_a_set_of_positive_measure} is the combination of the following two results (cf. the work of Rachid Regbaoui \cite[proof of Theorem 2.1]{regbaoui2001unique}). 
 \begin{enumerate}
  \item[(a)] $H^1(\Omega')$-solutions of \eqref{eq:Schroedinger_equation_UCPpos} that vanish on a measurable set of positive measure have zeros of infinite order (see, e.g.,
the result of de Figueiredo and Gossez \cite[Proposition 3]{de1992strict} or the result of Hadi and Tsouli \cite[Theorem 2.1]{hadi2001strong}).
  \item[(b)] The trivial solution $u=0$ is the only $H^1(\Omega')$-solution of \eqref{eq:Schroedinger_equation_UCPpos} that has a zero of infinite order (see, e.g, the book of H\"ormander
 \cite[Theorem 17.2.6]{hormander1983analysis}).
 \end{enumerate}
\end{proof}

\begin{definition}[Virtual measurement operators]\label{def:virtual_measurement_operators}
Let $B\subseteq V\cap\Omega$ be a non-empty closed ball with $q_1\gneq q_2$ on $B$ as in \eqref{eq:the_ball_B}.
 The operators $L_B$ and $L_{\Omega\setminus V}$ are defined by
 \begin{eqnarray}
  L_B:L^2(B)\to L^2(\Gamma),&\quad& f\mapsto v_B\vert_\Gamma,\\
  L_{\Omega\setminus V}:L^2(\Omega\setminus V)\to L^2(\Gamma), & \quad & h\mapsto v_{\Omega\setminus V}\vert_\Gamma,
 \end{eqnarray}
 where $v_B,v_{\Omega\setminus V}\in H^1(\Omega)$ are the unique solutions of
  \begin{eqnarray}
  -\Delta v_B+q_1v_B=|q_1-q_2|^{1/2}f\chi_B\ \text{in}\ \Omega &\text{with}& \partial_\nu v_B\vert_{\partial\Omega}=0,\\
  -\Delta v_{\Omega\setminus V}+q_1v_{\Omega\setminus V}=|q_1-q_2|^{1/2}h\chi_{\Omega\setminus V}\ \text{in}\ \Omega &\text{with}& \partial_\nu v_{\Omega\setminus V}\vert_{\partial\Omega}=0,
 \end{eqnarray}
or equivalently
 \begin{align}
  \myint{\Omega}{\nabla v_B\cdot\nabla w+q_1v_Bw}{x}&=\myint{B}{|q_1-q_2|^{1/2}wf}{x}\quad\forall w\in H^1 (\Omega),\\
  \myint{\Omega}{\nabla v_{\Omega\setminus V}\cdot\nabla w+q_1v_{\Omega\setminus V}w}{x}&=\myint{\Omega\setminus V}{|q_1-q_2|^{1/2}wh}{x}\quad\forall w\in H^1 (\Omega).
 \end{align}
\end{definition}

\begin{lemma}\label{lemma:well_definedness_and_some_properties_of_the_vm_operators}
\begin{enumerate}
 \item[(a)] The adjoint operators
 \begin{equation}
  L_B^\ast:L^2(\Gamma)\to L^2(B)\quad\text{and}\quad L_{\Omega\setminus V}^\ast:L^2(\Gamma)\to L^2(\Omega\setminus V)
 \end{equation}
fulfill
\begin{equation}
 L_B^\ast g=\left.\left(|q_1-q_2|^{1/2}u\right)\right\vert_B\quad\text{and}\quad L_{\Omega\setminus V}^\ast g=\left.\left(|q_1-q_2|^{1/2}u\right)\right\vert_{\Omega\setminus V},
\end{equation}
where $u:=u^{(g)}_{q_1}\in H^1(\Omega)$ is the corresponding solution of \eqref{eq:N_BVP_of_schroedinger_equation}.
\item[(b)] The adjoint operator $L_B^\ast$ is injective and $\overline{\mathcal{R}(L_B)}=L^2(\Gamma)$.
\item[(c)] $\mathcal{R}(L_B)\cap\mathcal{R}(L_{\Omega\setminus V})=\lbrace0\rbrace$.
\item[(d)] $\mathcal{R}(L_B)\not\subseteq\mathcal{R}(L_{\Omega\setminus V})$.
\item[(e)] $\not\exists C>0\,:\,\|L_B^\ast g\|\leq C\|L_{\Omega\setminus V}^\ast g\|\quad\forall g\in L^2(\Gamma)$.
\end{enumerate}
\end{lemma}

\begin{proof}
\begin{enumerate}
 \item[(a)] For $f\in L^2(B)$, let $v_B^{(f)}\in H^1(\Omega)$ be the solution of
 \begin{equation*}
  \myint{\Omega}{\nabla v_B^{(f)}\cdot\nabla w+q_1v_B^{(f)}w}{x}=\myint{B}{|q_1-q_2|^{1/2}wf}{x}\quad\forall w\in H^1(\Omega).
 \end{equation*}
Then, in particular,
$$L_Bf=v_B^{(f)}\vert_{\Gamma}.$$
 Furthermore, for $g\in L^2(\Gamma)$, let $u^{(g)}_{q_1}$ be the corresponding solution of \eqref{eq:N_BVP_of_schroedinger_equation}. Then, $u^{(g)}_{q_1}$ also solves the equivalent variational formulation
 \begin{equation*}
  \myint{\Omega}{\nabla w\cdot\nabla u^{(g)}_{q_1}+q_1wu^{(g)}_{q_1}}{x}=\myint{\Gamma}{g w\vert_\Gamma}{s}\quad\forall w\in H^1(\Omega).
 \end{equation*}
Hence, for arbitrary $f\in L^2(B)$ and $g\in L^2(\Gamma)$,
\begin{align*}
 (f,L_B^\ast g)_{L^2(B)}&=(L_Bf,g)_{L^2(\Gamma)}=\myint{\Gamma}{gv_B^{(f)}\vert_{\Gamma}}{s}\\
 &=\myint{\Omega}{\nabla v_B^{(f)}\cdot\nabla u^{(g)}_{q_1}+q_1v_B^{(f)}u^{(g)}_{q_1}}{x}\\
 &=\myint{B}{|q_1-q_2|^{1/2}u^{(g)}_{q_1}f}{x}\\
 &=\left(f,\left(|q_1-q_2|^{1/2}u^{(g)}_{q_1}\right)\vert_{B}\right)_{L^2(B)}.
\end{align*}
This yields $L_B^\ast g=\left.\left(|q_1-q_2|^{1/2}u^{(g)}_{q_1}\right)\right\vert_{B}$.

Analogously, it follows $L_{\Omega\setminus V}^\ast g=\left.\left(|q_1-q_2|^{1/2}u^{(g)}_{q_1}\right)\right\vert_{\Omega\setminus V}$.
\item[(b)] First, we prove the injectivity of $L_B^\ast$. Let $g\in L^2(\Gamma)$ with $L_B^\ast g=0$ and
$u:=u^{(g)}_{q_1}\in H^1(\Omega)$ be the corresponding solution of \eqref{eq:N_BVP_of_schroedinger_equation}.
>From (a) it follows $L_B^\ast g=\left(|q_1-q_2|^{1/2}u\right)\vert_{B}$.
%
Since $q_1-q_2\gneq 0$ on $B$, there exists a measurable set $E\subseteq B$ of positive measure where $|q_1-q_2|^{1/2}\neq0$.
 Hence, $\left(|q_1-q_2|^{1/2}u\right)\vert_E\equiv 0$ implies $u\vert_E\equiv 0$.
>From Theorem \ref{th:UCP_from_a_set_of_positive_measure} it follows that $u\equiv0$ on $\Omega$ and thus $g=\partial_\nu u\vert_\Gamma=0$.
%
This shows the injectivity of $L_B^\ast$ and thus $\overline{\mathcal{R}(L_B)}=\mathcal{N}(L_B^\ast)^\perp=L^2(\Gamma)$. 

\item[(c)] Recall that $B$ and $\Omega\setminus V$ are closed in $\Omega$ and that $V\setminus B$ is connected.
Let
\begin{equation*}
 \phi=L_B f=L_{\Omega\setminus V} h\in\mathcal{R}(L_B)\cap\mathcal{R}(L_{\Omega\setminus V})
\end{equation*}
and $v_B,v_{\Omega\setminus V}\in H^1(\Omega)$ be the corresponding solutions of Definition \ref{def:virtual_measurement_operators}.
First, we show that
\begin{equation}\label{eq:v1_equal_v2}
 v_B=v_{\Omega\setminus V}\quad\text{on}\quad\Omega\setminus\left(\overline{B}\cup\overline{\Omega\setminus V}\right)=(\Omega\cap V)\setminus B:
\end{equation} 
On $\Omega\cup V$, we define the continuations

\begin{equation*}
 q:=\begin{cases}
    q_1, & \text{on }\Omega,\\
    1,   & \text{on }V\setminus\Omega,
   \end{cases}
\end{equation*}
\begin{equation*}
 \tilde v:=\begin{cases}
    v, & \text{on }\Omega,\\
    0,   & \text{on }V\setminus\Omega,
   \end{cases}
   \quad\text{and}\quad
   \tilde v_j:=\begin{cases}
    \partial_{x_j} v, & \text{on }\Omega,\\
    0,   & \text{on }V\setminus\Omega,
   \end{cases}
\end{equation*}
where $v:=v_B-v_{\Omega\setminus V}$.

Obviously, $\tilde v,\tilde v_j\in L^2(\Omega\cup V)$. To verify that $\tilde v\in H^1(\Omega\cup V)$, it is left to show $\partial_{x_j}\tilde v=\tilde v_j$. This can be shown by using
\begin{equation*}
 v\vert_\Gamma=v_B\vert_\Gamma-v_{\Omega\setminus V}\vert_\Gamma=\phi-\phi=0.
\end{equation*}
Let $\varphi\in\mathcal{D}(\Omega\cup V),$ then,
\begin{align*}
 \myint{\Omega\cup V}{\tilde v \partial_{x_j}\varphi}{x}&=\myint{\Omega}{v \partial_{x_j}\varphi}{x}\\
 &=-\myint{\Omega}{\varphi\partial_{x_j}v}{x}+\myint{\partial\Omega}{(\varphi v)\vert_\Gamma\nu_j}{s}\\
 &=-\myint{\Omega}{v_j\varphi}{x}=-\myint{\Omega\cup V}{\tilde v_j\varphi}{x}.
\end{align*}

Now, we go on showing \eqref{eq:v1_equal_v2}. Since $v=\tilde v\vert_\Omega$ fulfills
\begin{align*}
  &\myint{\Omega}{\nabla v\cdot\nabla w+q_1vw}{x}\\
  &\quad\quad=\myint{\Omega}{|q_1-q_2|^{1/2}w(f\chi_B-h\chi_{\Omega\setminus V})}{x}\quad\forall w\in H^1 (\Omega),
\end{align*}
it holds
\begin{equation*}
  \myint{V\setminus B}{\nabla \tilde v\cdot\nabla \varphi+q\tilde v\varphi}{x}=0\quad\forall \varphi\in \mathcal{D}\left(V\setminus B\right).
\end{equation*}
We obtain that $\tilde v$ (as a function in $H^1\left(V\setminus B\right)$)
solves
\begin{equation*}
 -\Delta \tilde v + q \tilde v=0\quad\text{on}\quad V\setminus B
\end{equation*}
and vanishes on $V\setminus\Omega$.
Since $V\setminus\Omega$ is a non-empty open set ($V$ is open and has a non-empty intersection $\Gamma$ with the Lipschitz-domain $\Omega$) and
$V\setminus B$ is connected,
Theorem \ref{th:UCP_from_a_set_of_positive_measure} shows that $\tilde v\equiv 0$ on
$V\setminus B$ and thus
\begin{equation*}
 v_B=v_{\Omega\setminus V}\quad\text{on}\quad (V\cap\Omega)\setminus B.
\end{equation*}

To finally show $\phi=0$, we define
\begin{equation*}
 u:=\begin{cases}
     v_B&\text{on}\ \Omega\setminus B,\\
     v_{\Omega\setminus V}&\text{on}\ B.
    \end{cases}
\end{equation*}
We can partition test functions (in $\mathcal{D}(\Omega)$ and $H^1(\Omega)$), by using smooth partitions of unity, to prove that $u$ is an $H^1(\Omega)$-function and the unique solution of
\begin{align*}
 -\Delta u+q_1u&=0\quad\text{on}\quad\Omega,\\
 \partial_\nu u\vert_{\partial\Omega}&=0.
\end{align*}
Hence, $u$ has to be equal to the trivial solution and thus
\begin{equation*}
 \phi=v_B\vert_\Gamma=u\vert_\Gamma\equiv0.
\end{equation*}
\item[(d)] This simply follows from (b) and (c).
\item[(e)] Let us assume there exists a constant $C>0$ such that $$\|L_B^\ast g\|\leq C\|L_{\Omega\setminus V}^\ast g\|\quad \forall g\in L^2(\Gamma).$$ Then,
 $$\mathcal{R}(L_B)\subseteq\mathcal{R}(L_{\Omega\setminus V})$$ immediately follows from Theorem \ref{theorem:14th_importent_propertie} and this is a contradiction to (d).
\end{enumerate}
\end{proof}

\begin{proof}[Proof of Lemma \ref{lemma:localized_potentials}]
 The assertion follows from Lemma \ref{lemma:well_definedness_and_some_properties_of_the_vm_operators}:

The trivial case is that when $L_{\Omega\setminus V}^\ast$ is not injective. Then, there exists an element $g\in L^2(\Gamma)\setminus\lbrace 0\rbrace$ with $\|L_{\Omega\setminus V}^\ast g\|=0$.
By the injectivity of $L_B^\ast$ we have $\|L_B^\ast g\|=:c_g\geq 0$. In this case, we can set $g_m:=mg$ for all $m\in\mathbb{N}$.

When $L_{\Omega\setminus V}^\ast$ is injective, we derive a suitable sequence $(g_m)_{m\in\mathbb{N}}\subseteq L^2(\Gamma)$ as follows.
 
 Let $C_m=m^2$ for $m\in\mathbb{N}$. Lemma \ref{lemma:well_definedness_and_some_properties_of_the_vm_operators} (e) implies the existence of a sequence $(g_m')_{m\in\mathbb{N}}\subseteq L^2(\Gamma)$ with
 \begin{equation}\label{inequality:converse_inequality_for_adjoint_vm_operators_for_C_m}
  \|L_B^\ast g_m'\|>C_m\|L_{\Omega\setminus V}^\ast g_m'\|\quad\forall m\in\mathbb{N}.
 \end{equation}

 In particular, this implies $g_m'\neq 0$ for all $m\in\mathbb{N}$. Since
 $L_{\Omega\setminus V}^\ast$ is injective, we can set $g_m:=\frac{g_m'}{m\|L_{\Omega\setminus V}^\ast g_m'\|}$. By multiplying \eqref{inequality:converse_inequality_for_adjoint_vm_operators_for_C_m} with $\frac{1}{m\|L_{\Omega\setminus V}^\ast g_m'\|}$, we obtain
 \begin{equation*}
  \|L_B^\ast g_m\|>m\quad\forall m\in\mathbb{N}.
 \end{equation*}
Furthermore, it holds
\begin{equation*}
 \|L_{\Omega\setminus V}^\ast g_m\|=\frac{1}{m}\quad\forall m\in\mathbb{N}.
\end{equation*}

For both cases, we obtain a sequence $(g_m)_{m\in\mathbb{N}}$ such that
\begin{align*}
 \lim_{m\to\infty}\myint{V\cap\Omega}{(q_1-q_2){u_m}^2}{x}=\lim_{m\to\infty}\|L_B^\ast g_m\|^2&=\infty,\\
 \lim_{m\to\infty}\myint{\Omega\setminus V}{(q_1-q_2){u_m}^2}{x}=\lim_{m\to\infty}\|L_{\Omega\setminus V}^\ast g_m\|^2&=0
\end{align*}
where $u_m:=u^{(g_m)}_{q_1}\in H^1(\Omega)$ is the corresponding solution of \eqref{eq:N_BVP_of_schroedinger_equation}.
\end{proof}

\bibliographystyle{amsplain}
\bibliography{literaturliste}

\providecommand{\bysame}{\leavevmode\hbox to3em{\hrulefill}\thinspace}
\providecommand{\MR}{\relax\ifhmode\unskip\space\fi MR }
\providecommand{\MRhref}[2]{%
  \href{http://www.ams.org/mathscinet-getitem?mr=#1}{#2}
}
\providecommand{\href}[2]{#2}
\begin{thebibliography}{10}

\bibitem{astala2006calderon}
Kari Astala and Lassi P{\"a}iv{\"a}rinta, \emph{Calder{\'o}n's inverse
  conductivity problem in the plane}, Ann. of Math. (2006), 265--299.

\bibitem{bourbaki2003elements}
Nicolas Bourbaki, \emph{Elements of mathematics: Topological vector spaces,
  chapters 1-5}, Springer-Verlag, 2003.

\bibitem{bukhgeim2008recovering}
AL~Bukhgeim and AA~Bukhgeim, \emph{Recovering a potential from cauchy data in
  the two-dimensional case}, J. Inverse Ill-Posed Probl \textbf{16} (2008),
  no.~1, 19--33.

\bibitem{calderon1980inverse}
Alberto~P Calder\'on, \emph{On an inverse boundary value problem}, Seminar on
  Numerical Analysis and its Application to Continuum Physics (W~H Meyer and
  M~A Raupp, eds.), Brasil. Math. Soc., Rio de Janeiro, 1980, pp.~65--73.

\bibitem{calderon2006inverse}
\bysame, \emph{On an inverse boundary value problem}, Comput. Appl. Math.
  \textbf{25} (2006), no.~2--3, 133--138.

\bibitem{de1992strict}
Djairo~G de~Figueiredo and Jean-Pierre Gossez, \emph{Strict monotonicity of
  eigenvalues and unique continuation}, Comm. Partial Differential Equations
  \textbf{17} (1992), no.~1-2, 339--346.

\bibitem{fruhauf2007detecting}
Florian Fr{\"u}hauf, Bastian Gebauer, and Otmar Scherzer, \emph{Detecting
  interfaces in a parabolic-elliptic problem from surface measurements}, SIAM
  J. Numer. Anal. \textbf{45} (2007), no.~2, 810--836.

\bibitem{gebauer2008localized}
Bastian Gebauer, \emph{Localized potentials in electrical impedance
  tomography}, Inverse Probl. Imaging \textbf{2} (2008), no.~2, 251--269.

\bibitem{haberman2013uniqueness}
Boaz Haberman, Daniel Tataru, et~al., \emph{Uniqueness in calder{\'o}n’s
  problem with lipschitz conductivities}, Duke Mathematical Journal
  \textbf{162} (2013), no.~3, 497--516.

\bibitem{hadi2001strong}
Islam~Eddine Hadi and N~Tsouli, \emph{Strong unique continuation of
  eigenfunctions for p-laplacian operator}, Int. J. Math. Math. Sci.
  \textbf{25} (2001), no.~3, 213--216.

\bibitem{harrach2009uniqueness}
Bastian Harrach, \emph{On uniqueness in diffuse optical tomography}, Inverse
  Problems \textbf{25} (2009), no.~5, 055010 (14pp).

\bibitem{harrach2012simultaneous}
\bysame, \emph{Simultaneous determination of the diffusion and absorption
  coefficient from boundary data}, Inverse Probl. Imaging \textbf{6} (2012),
  no.~4, 663--679.

\bibitem{harrach2013monotonicity}
Bastian Harrach and Marcel Ullrich, \emph{Monotonicity-based shape
  reconstruction in electrical impedance tomography}, SIAM J. Math. Anal.
  \textbf{45} (2013), no.~6, 3382--3403.

\bibitem{hormander1983analysis}
Lars H{\"o}rmander, \emph{The analysis of linear partial differential operators
  {III}}, vol. 274, Springer, 1994.

\bibitem{ikehata1998size}
M.~Ikehata, \emph{Size estimation of inclusion}, J. Inverse Ill-Posed Probl.
  \textbf{6} (1998), no.~2, 127--140.

\bibitem{imanuvilov2010calderon}
Oleg Imanuvilov, Gunther Uhlmann, and Masahiro Yamamoto, \emph{The calder{\'o}n
  problem with partial data in two dimensions}, Journal of the American
  Mathematical Society \textbf{23} (2010), no.~3, 655--691.

\bibitem{isakov2007uniqueness}
Victor Isakov, \emph{On uniqueness in the inverse conductivity problem with
  local data}, Inverse Problems and Imaging \textbf{1} (2007), no.~1, 95.

\bibitem{Kan97}
Hyeonbae Kang, Jin~Keun Seo, and Dongwoo Sheen, \emph{The inverse conductivity
  problem with one measurement: stability and estimation of size}, SIAM J.
  Math. Anal. \textbf{28} (1997), no.~6, 1389--1405.

\bibitem{kenig2012calderon}
Carlos~E Kenig and Mikko Salo, \emph{The calder{\'o}n problem with partial data
  on manifolds and applications}, arXiv preprint arXiv:1211.1054 (2012).

\bibitem{kenig2007calderon}
Carlos~E Kenig, Johannes Sj{\"o}strand, and Gunther Uhlmann, \emph{The
  calder{\'o}n problem with partial data}, Annals of mathematics (2007),
  567--591.

\bibitem{kenig2014recent}
CE~Kenig and Mikko Salo, \emph{Recent progress in the calder{\'o}n problem with
  partial data}, Contemp. Math \textbf{615} (2014), 193--222.

\bibitem{kohn1984determining}
Robert Kohn and Michael Vogelius, \emph{Determining conductivity by boundary
  measurements}, Comm. Pure Appl. Math. \textbf{37} (1984), no.~3, 289--298.

\bibitem{kohn1985determining}
Robert~V Kohn and Michael Vogelius, \emph{Determining conductivity by boundary
  measurements {II}. interior results}, Comm. Pure Appl. Math. \textbf{38}
  (1985), no.~5, 643--667.

\bibitem{nachman2010reconstruction}
Adrian Nachman and Brian Street, \emph{Reconstruction in the calder{\'o}n
  problem with partial data}, Communications in Partial Differential Equations
  \textbf{35} (2010), no.~2, 375--390.

\bibitem{nachman1996global}
Adrian~I Nachman, \emph{Global uniqueness for a two-dimensional inverse
  boundary value problem}, Ann. of Math. (1996), 71--96.

\bibitem{regbaoui2001unique}
Rachid Regbaoui, \emph{Unique continuation from sets of positive measure},
  Carleman Estimates and Applications to Uniqueness and Control Theory,
  Springer, 2001, pp.~179--190.

\bibitem{sylvester1987global}
John Sylvester and Gunther Uhlmann, \emph{A global uniqueness theorem for an
  inverse boundary value problem}, Annals of mathematics (1987), 153--169.

\end{thebibliography}

\end{document}